\documentclass{article}

\usepackage{amsmath, amsfonts, amssymb, amsthm, mathtools}
\usepackage{graphicx}

\newtheorem{theoremA}{Theorem}

\newtheorem{theorem}{Theorem}
\newtheorem*{corollary*}{Corollary}
\newtheorem{lemma}{Lemma}
\newtheorem{lemmaA}{Lemma}

\numberwithin{equation}{section}

\title{A refinement of a Hardy type inequality for negative exponents, and sharp applications to Muckenhoupt weights on $\mathbb R$}
\author{Eleftherios N. Nikolidakis \and Theodoros Stavropoulos}

\begin{document}
\maketitle

\begin{abstract}
We prove a sharp integral inequality that generalizes the well known Hardy type integral inequality for negative exponents. We also give sharp applications
in two directions for Muckenhoupt weights on $\mathbb R$. This work refines the results that appear in \cite{9}.
\end{abstract}

\section{Introduction} \label{sect:1}
In 1920, Hardy has proved (as one can see in \cite{2} or \cite{3}) the following inequality which is known as Hardy's inequality
\begin{theoremA} \label{thm:A}
If $p>1$, $a_n\geq 0$ and $A_n=a_1+a_2+\ldots+a_n$, $n\in \mathbb N^*$, then
\begin{equation} \label{eq:1p1}
\sum_{n=1}^\infty \left(\frac{A_n}{n}\right)^p \leq \left(\frac{p}{p-1}\right)^p\sum_{n=1}^\infty a^p_n.
\end{equation}
Moreover, inequality \eqref{eq:1p1} is best possible, that is the constant on the right side cannot be decreased.
\end{theoremA}

In 1926, Copson generalized in \cite{1} Theorem \ref{thm:A}, by replacing the arithmetic mean of a sequence by a weighted arithmetic mean. More precisely, he proved the following

\begin{theoremA} \label{thm:B}
Let $p>1,\ a_n,\lambda_n>0$ for $n=1, 2, \ldots$. Further suppose that $\Lambda_n=\sum_{i=1}^n\lambda_i$ and $A_n=\sum_{i=1}^n\lambda_ia_i$. Then
\begin{equation} \label{eq:1p2}
\sum_{n=1}^\infty\lambda_n\left(\frac{A_n}{\Lambda_n}\right)^p \leq \left(\frac{p}{p-1}\right)^p\sum_{n=1}^\infty\lambda_na_n^p,
\end{equation}
where the constant involved in \eqref{eq:1p2} is best possible.
\end{theoremA}

Certain generalizations of \eqref{eq:1p1} have been given in \cite{6}, \cite{7} and elsewhere. For example, one can see in \cite{8} further generalizations of Hardy's and Copson's inequalities be replacing means by more general linear transforms. Theorem \ref{thm:A} has a continued analogue which is the following

\begin{theoremA}  \label{thm:C}
If $p>1$ and $f: [0,+\infty)\to\mathbb R^+$ is $L^p$-integrable, then
\begin{equation} \label{eq:1p3}
\int_0^\infty \left(\frac 1 t \int_0^t f(u)\,\mathrm du\right)^p\mathrm dt \leq \left(\frac{p}{p-1}\right)^p\int_0^\infty f^p(t)\,\mathrm dt.
\end{equation}
The constant in the right side of \eqref{eq:1p3} is best possible.
\end{theoremA}

It is easy to see that Theorems \ref{thm:A} and \ref{thm:C} are equivalent, by standard approximation arguments which involve step functions. Now as one can see in \cite{4}, there is a continued analogue of \eqref{eq:1p3} for negative exponents, which is presented there without a proof. This is described in the following

\begin{theoremA} \label{thm:D}
Let $f: [a,b]\to\mathbb R^+$. Then for every $p>0$ the following is true
\begin{equation} \label{eq:1p4}
\int_a^b\left(\frac{1}{t-a}\int_a^tf(u)\,\mathrm du\right)^{-p}\mathrm dt \leq \left(\frac{p+1}{p}\right)^p\int_a^bf^{-p}(t)\,\mathrm dt.
\end{equation}
Moreover \eqref{eq:1p4} is best possible.
\end{theoremA}

In \cite{9}, a generalization of \eqref{eq:1p4} has been given, which can be seen in the following
\begin{theoremA} \label{thm:E}
Let $p\geq q>0$ and $f: [a,b]\to\mathbb R^+$. The following inequality is true and sharp
\begin{equation} \label{eq:1p5}
\int_a^b\left(\frac{1}{t-a}\int_a^tf(u)\,\mathrm du\right)^{-\mathrlap{p}}\mathrm dt \leq \left(\frac{p+1}{p}\right)^q\int_a^b\left(\frac{1}{t-a}\int_a^tf(u)\,\mathrm du\right)^{-p\mathrlap{+q}}f^{-q}(t)\,\mathrm dt.
\end{equation}
\end{theoremA}

What is proved in fact in \cite{9} is a more general weighted discrete analogue of \eqref{eq:1p5} which is given in the following
\begin{theoremA} \label{thm:F}
Let $p\geq q>0$ and $a_n,\lambda_n>0$ for $n=1, 2, \ldots$. Define $A_n$ and $\Lambda_n$ as in Theorem \ref{thm:B}. Then
\begin{equation} \label{eq:1p6}
\sum_{n=1}^\infty\lambda_n\left(\frac{A_n}{\Lambda_n}\right)^{-p} \leq \left(\frac{p+1}{p}\right)^q\sum_{n=1}^\infty \lambda_n\left(\frac{A_n}{\Lambda_n}\right)^{-p+q}a_n^{-q}.
\end{equation}
\end{theoremA}

Certain applications exist for the above two theorems. One of them can be seen in \cite{9}, concerning Muckenhoupt weights. In this paper we generalize and refine inequality \eqref{eq:1p5} by specifying the integral of $f$ over $[a,b]$. We also assume, for simplicity reasons, that $f$ is Riemann integrable on $[a,b]$. More precisely we will prove the following
\begin{theorem} \label{thm:1}
Let $p\geq q>0$ and $f: [a,b]\to \mathbb R^+$ with $\frac{1}{b-a}\int_a^bf=\ell$. Then the following inequality is true
\begin{multline} \label{eq:1p7}
\int_a^b\left(\frac{1}{t-a}\int_a^tf(u)\,\mathrm du\right)^{\mathrlap{-p}}\mathrm dt \leq \left(\frac{p+1}{p}\right)^q\int_a^b\left(\frac{1}{t-a}\int_a^tf(u)\,\mathrm du\right)^{-\mathrlap{p+q}}f^{-q}(t)\,\mathrm dt - \\
- \frac{q}{p+1}(b-a)\cdot\ell^{-p}.
\end{multline}
Moreover, inequality \eqref{eq:1p7} is sharp if one considers all weights $f$ that have mean integral average over $[a,b]$ equal to $\ell$.
\end{theorem}

What we mean by noting that \eqref{eq:1p7} is sharp is the following: The constant in front of the integral on the right side cannot be decreased, while the one in front of $\ell^{-p}$ cannot be increased. These facts will be proved below. In fact more is true as can be seen in the following
\begin{theorem} \label{thm:2}
Let $p\geq q>0$ and $a_n,\lambda_n>0$, for every $n=1, 2, \ldots$. Define $A_n$ and $\Lambda_n$ as above. Then the following inequality holds for every $N\in\mathbb N$.
\begin{equation} \label{eq:1p8}
\sum_{n=1}^N\lambda_n\left(\frac{A_n}{\Lambda_n}\right)^{-p} \leq \left(\frac{p+1}{p}\right)^q\sum_{n=1}^N\lambda_n\left(\frac{A_n}{\Lambda_n}\right)^{-p\mathrlap{+q}}a_n^{-q} - \frac{q}{p+1}\Lambda_N\left(\frac{A_N}{\Lambda_N}\right)^{-p}.
\end{equation}
\end{theorem}

In Section \ref{sect:2} we describe the proof of Theorem \ref{thm:2} and we also prove the validity and the sharpness of \eqref{eq:1p7}. Moreover if one wants to study the whole topic concerning generalization of inequalities \eqref{eq:1p1} or \eqref{eq:1p2}, can see \cite{5} and \cite{10}.

In the last section we prove an application of Theorem \ref{thm:1}. More precisely we prove the following
\begin{theorem} \label{thm:3}
Let $\varphi: [0,1)\to\mathbb R^+$ be non-decreasing satisfying the following Muckenhoupt type inequality
\begin{equation} \label{eq:1p9}
\left(\frac 1 t \int_0^t\varphi(y)\,\mathrm dy\right)\left(\frac 1 t \int_0^t\varphi^{-1/(q-1)}(y)\,\mathrm dy\right)^{q-1}\leq M,
\end{equation}
for every $t\in(0,1]$, where $q>1$ is fixed and $M\geq 1$ is given. Let now $p_0\in(1,q)$ be defined as the solution of the following equality:
\begin{equation} \label{eq:1p10}
\frac{q-p_0}{q-1}(M\,p_0)^{1/(q-1)} = 1\,.
\end{equation}
Then for every $p\in(p_0,q]$ the following inequality
\begin{equation} \label{eq:1p11}
\frac 1 t \int_0^t\left(\frac 1 s \int_0^s \varphi\right) ^{-1/(p-\mathrlap{1)}}\mathrm ds \leq \left(\frac 1 t \int_0^t\varphi\right)^{-1/(p-1\mathrlap{)}}\frac{1}{K'}\,c\,\frac q p \left(\frac{p-1}{q-1}\right)^2
\end{equation}
is true, for every $t\in (0,1]$, where $c=M^{1/(q-1)}$ and $K' =K'(p,q,c) = \dfrac{1}{p^{1/(q-1)}}-c\,\dfrac{q-p}{q-1}$.
It is also sharp for $t=1$.
\end{theorem}

The above theorem implies immediately the following
\begin{corollary*} Let $\varphi$ be as in Theorem \ref{thm:3}. Then the following inequality is true for every $t\in (0,1]$ and every $p\in(p_0,q]$.
\[
\left(\frac 1 t \int_0^t\varphi^{-1/(p-1)}\right)^{p-1}\left(\frac 1 t \int_0^t\varphi\right) \leq \left[\frac{1}{K'}\,c\,\frac q p \left(\frac{p-1}{q-1}\right)^2\right]^{p-1}
\]
This gives us the best possible range of $p$'s for which the Muckenhoupt condition \eqref{eq:1p9} still holds, under the hypothesis of \eqref{eq:1p9}.
\end{corollary*}
The above corollary is the content of \cite{9} but with another constant. Thus by proving Theorem 3 we refine the results in \cite{9} by improving
the constants that appear there and by giving certain sharp inequalities that involve Muckenhoupt weights on $\mathbb R$.

\section{The Hardy inequality} \label{sect:2}

\begin{proof}[Proof of Theorem \ref{thm:2}]~

Let $p\geq q>0$ and $a_n,\lambda_n>0$, for every $n\in\mathbb N^*$. We define $\Lambda_n=\lambda_1+\lambda_2+\ldots+\lambda_n$, $A_n=\lambda_1a_1+\lambda_2a_2+\ldots+\lambda_na_n$, for $n=1, 2, \ldots$. We shall prove inequality \eqref{eq:1p8}. In order to do this we will give two Lemmas that are stated below. We follow \cite{9}.

\begin{lemma} \label{lem:1}
Under the above notation the following inequality holds for every $n\in\mathbb N^*$.
\begin{equation} \label{eq:2p1}
\left(\frac{p+1}{p}\right)^qa_n^{-q}\left(\frac{A_n}{\Lambda_n}\right)^{-p\mathrlap{+q}}+p\left(\frac{p}{p+1}\right)^{q/p}a_n^{q/p}\left(\frac{A_n}{\Lambda_n}\right)^{-p-q/p} \geq (p+1)\left(\frac{A_n}{\Lambda_n}\right)^{-p}.
\end{equation}
\end{lemma}

\begin{proof}
It is well known that the following inequality holds
\begin{equation} \label{eq:2p2}
y_1^{-p} + p\,y_1\,y_2^{-p-1}-(p+1)\,y_2^{-p} \geq 0,
\end{equation}
for every $y_1, y_2 > 0$.

This is in fact an immediate consequence of the inequality
\begin{equation} \label{eq:2p3}
y^{-p} + p\,y \geq (p+1),\ \text{for every}\ y, p\geq 0.
\end{equation}

Inequality \eqref{eq:2p3} is true in view of Young's inequality which asserts that for every $t,s$ nonnegative the following inequality is true
\begin{equation}
\frac{1}{q}t^q+\frac{1}{q^{'}}s^{q{'}}\geq ts
\end{equation}

whenever $q$ is greater than 1, and $q^{'}$ is such that $\frac{1}{q}+\frac{1}{q^{'}}=1$. Then by choosing $q=p+1$ in Young's inequality, and
setting $t=\frac{1}{y}$ we obtain \eqref{eq:2p3}.

If we apply \eqref{eq:2p3} when $y=y_1/y_2$ we obtain \eqref{eq:2p2}. Now we apply \eqref{eq:2p2} when $\displaystyle y_1 = \left(\frac{p}{p+1}\right)^{1+q\mathrlap{/p}} a_n^{q/p} \left(\frac{A_n}{\Lambda_n}\right)^{1-q\mathrlap{/p}}$ and $\displaystyle y_2 = \left(\frac{p}{p+1}\right)\frac{A_n}{\Lambda_n}$.

Then as it is easily seen \eqref{eq:2p1} is immediately proved. Our proof of Lemma \ref{lem:1} is now complete.
\end{proof}

As a consequence of Lemma \ref{lem:1} we have (by summing the respective inequalities) that:
\begin{multline} \label{eq:2p4}
\left(\frac{p+1}{p}\right)^q\sum_{n=1}^N\lambda_na_n^{-q}\left(\frac{A_n}{\Lambda_n}\right)^{-p\mathrlap{+q}} + p\left(\frac{p}{p+1}\right)^{q/p}\sum_{n=1}^N\lambda_na_n^{q/p}\left(\frac{A_n}{\Lambda_n}\right)^{-p-q/p} \\
\geq (p+1)\sum_{n=1}^N\left(\frac{A_n}{\Lambda_n}\right)^{-p}\lambda_n,
\end{multline}
for every $N\in\mathbb N^*$.

We proceed to the proof of
\begin{lemma} \label{lem:2}
Under the above notation the following inequality is true for every $N\in \mathbb N^*$
\begin{equation} \label{eq:2p5}
\sum_{n=1}^N\lambda_n\left(\frac{A_n}{\Lambda_n}\right)^{-p} - \left(\frac{p}{p+1}\right)\sum_{n=1}^N\lambda_na_n\left(\frac{A_n}{\Lambda_n}\right)^{-p-1} \geq \frac{\Lambda_n}{p+1}\left(\frac{A_n}{\Lambda_n}\right)^{-p}.
\end{equation}
\end{lemma}

\begin{proof}
We follow \cite{9}.

For $N=1$, inequality \eqref{eq:2p5} is in fact equality. We suppose now that it is true with $N-1$ in place of $N$. We will prove that it is also true for the choice of $N$.

\begin{multline} \label{eq:2p6}
\text{Define}\ S_N = \sum_{n=1}^N \left[\lambda_n\left(\frac{A_n}{\Lambda_n}\right)^{-p}-\left(\frac{p}{p+1}\right)\lambda_n\,a_n\left(\frac{A_n}{\Lambda_n}\right)^{-p-1}\right] = \\
\hspace{1em}\begin{multlined}
= \sum_{n=1}^{N-1}\left[\lambda_n\left(\frac{A_n}{\Lambda_n}\right)^{-p}-\left(\frac{p}{p+1}\right)\lambda_n\,a_n\left(\frac{A_n}{\Lambda_n}\right)^{-p-1}\right] + \\
+ \lambda_N\left(\frac{A_N}{\Lambda_N}\right)^{-p} - \left(\frac{p}{p+1}\right)(A_N-A_{N-1})\left(\frac{A_N}{\Lambda_N}\right)^{-p-1}.
\end{multlined}
\end{multline}
By our induction step we obviously see that
\begin{equation} \label{eq:2p7}
\begin{aligned}
S_N & \geq \frac{\Lambda_{N-1}}{p+1}\left(\frac{A_{N-1}}{\Lambda_{N-1}}\right)^{-p} + \lambda_N\left(\frac{A_N}{\Lambda_N}\right)^{-p} - \left(\frac{p}{p+1}\right)(A_N{-}A_{N-1})\left(\frac{A_N}{\Lambda_N}\right)^{-\mathrlap{p-1}} = \\
 & \begin{multlined}[.9\textwidth]
 = \frac{\Lambda_{N-1}}{p+1}\left(\frac{A_{N-1}}{\Lambda_{N-1}}\right)^{-p} + \lambda_N\left(\frac{A_N}{\Lambda_N}\right)^{-p} - \frac{p}{p+1}\Lambda_N\left(\frac{A_N}{\Lambda_N}\right)^{-p} + \\
 + \frac{\Lambda_{N-1}}{p+1}\left[p\,\frac{A_{N-1}}{\Lambda_{N-1}}\left(\frac{A_N}{\Lambda_N}\right)^{-p-1}\right].
\end{multlined}
\end{aligned}
\end{equation}
We use now inequality \eqref{eq:2p2} in order to find a lower bound for the expression in brackets in \eqref{eq:2p7}. We thus have
\begin{equation} \label{eq:2p8}
p\left(\frac{A_{N-1}}{\Lambda_{N-1}}\right)\left(\frac{A_N}{\Lambda_N}\right)^{-p-1} \geq -\left(\frac{A_{N-1}}{\Lambda_{N-1}}\right)^{-p} + (p+1)\left(\frac{A_N}{\Lambda_N}\right)^{-p}.
\end{equation}

We use \eqref{eq:2p8} in \eqref{eq:2p7} and obtain that
\begin{align*}
S_N \geq & \begin{multlined}[.9\textwidth][t]\frac{\Lambda_{N-1}}{p+1}\left(\frac{A_{N-1}}{\Lambda_{N-1}}\right)^{-p} + \lambda_N\left(\frac{A_N}{\Lambda_N}\right)^{-p}-\left(\frac{p}{p+1}\right)\Lambda_N\left(\frac{A_N}{\Lambda_N}\right)^{-p} + \\
+ \frac{\Lambda_{N-1}}{p+1}\left[(p+1)\left(\frac{A_N}{\Lambda_N}\right)^{-p} - \left(\frac{A_{N-1}}{\Lambda_{N-1}}\right)^{-p}\right] = \end{multlined} \\
 = & \left(\frac{A_N}{\Lambda_N}\right)^{-p}\left(\lambda_N-\frac{p}{p+1}\Lambda_N+\Lambda_{N-1}\right) = \frac{\Lambda_N}{p+1}\left(\frac{A_N}{\Lambda_N}\right)^{-p}
\end{align*}
that is \eqref{eq:2p5} holds. In this way we derived inductively the proof of our Lemma.
\end{proof}
We consider now the quantity
\begin{equation} \label{eq:2p9}
y = \sum_{n=1}^N \lambda_na_n^{q/p}\left(\frac{A_n}{\Lambda_n}\right)^{-p-q/p}.
\end{equation}
Then $\displaystyle y = \sum_{n=1}^N\lambda_n \left[a_n^{q/p}\left(\frac{A_n}{\Lambda_n}\right)^{-q-q/p}\right]\left[\frac{A_n}{\Lambda_n}\right]^{-p+q}$.
Suppose that $p>q$. The case $p=q$ will be discussed in the end of the proof.
Applying H\"older's inequality now in the above sum with exponents $r=\frac p q$ and $r'=\frac{p}{p-q}$, we have as a consequence that
\begin{align}
y \leq & \left\{\sum_{n=1}^N\lambda_na_n\left(\frac{A_n}{\Lambda_n}\right)^{-p-1}\right\}^{\frac q p} \left\{\sum_{n=1}^N\lambda_n\left(\frac{A_n}{\Lambda_n}\right)^{-p}\right\}^{1-\frac q p} \nonumber \\
 \leq & \left\{\frac{p+1}{p}\sum_{n=1}^N\lambda_n\left(\frac{A_n}{\Lambda_n}\right)^{-p} - \frac 1 p \Lambda_N\left(\frac{A_N}{\Lambda_N}\right)^{-p}\right\}^{\mathrlap{\frac q p}}\cdot \left\{\sum_{n=1}^N\lambda_n\left(\frac{A_n}{\Lambda_n}\right)^{-p}\right\}^{1-\frac q p}, \label{eq:2p10}
\end{align}
in view of Lemma \ref{lem:2}.

We set now $z=\sum_{n=1}^N \lambda_na_n^{-q}\left(\frac{A_n}{\Lambda_n}\right)^{-p+q}$ and $x = \sum_{n=1}^N\lambda_n\left(\frac{A_n}{\Lambda_n}\right)^{-p}$. Because of \eqref{eq:2p10} we have that
\begin{equation} \label{eq:2p11}
y\leq \left\{\frac{p+1}{p} x-\frac 1 p \Lambda_N\left(\frac{A_N}{\Lambda_N}\right)^{-p}\right\}^{\mathrlap{\frac q p }}\cdot x^{1-\frac q p}.
\end{equation}
By setting now $c=\Lambda_N\left(\frac{A_N}{\Lambda_N}\right)^{-p}$, we have because of \eqref{eq:2p11},
\begin{equation} \label{eq:2p12}
y\leq \left(\frac{p+1}{p}x - \frac c p\right)^{\mathrlap{\frac q p}}\cdot x^{1-\frac q p} = \left(\frac{p+1}{p}\right)^{\frac q p}\left[x-\frac{c}{p+1}\right]^{\frac q p} x^{1-\frac q p}.
\end{equation}
Note that by \eqref{eq:2p12} the quantity $x-\frac{c}{p+1}$ is positive, that is $x>\frac{c}{p+1}$. Now because of Lemma 1, it is immediate that
\begin{gather}
\left(\frac{p+1}{p}\right)^qz + p\left(\frac{p}{p+1}\right)^{\frac q p}y \geq (p+1)x \overset{\eqref{eq:2p12}}{\implies} \nonumber \\
\left(\frac{p+1}{p}\right)^qz + p\left(\frac{p}{p+1}\right)^{\frac q p} \left(\frac{p+1}{p}\right)^{\frac q p} \left[x-\frac{c}{p+1}\right]^{\frac q p} x^{1-\frac q p} \geq (p+1)x \implies \nonumber \\
\left(\frac{p+1}{p}\right)^qz \geq (p+1)x - p\left[x-\frac{c}{p+1}\right]^{\frac q p} x^{1-\frac q p} = \nonumber \\
= x + \left\{p\,x-p\left[x-\frac{c}{p+1}\right]^{\frac q p}x^{1-\frac q p}\right\} = x + p\,G(x), \label{eq:2p13}
\end{gather}
where $G(x)$ is defined for $x>\frac{c}{p+1}$, by $G(x) = x-\left[x-\frac{c}{p+1}\right]^{\frac qp}x^{1-\frac q p}$.

By \eqref{eq:2p13} now we obtain
\begin{equation} \label{eq:2p14}
\left(\frac{p+1}{p}\right)^qz-x \geq p\,G(x) \geq p \inf\left\{G(x): x>\frac{c}{p+1}\right\},
\end{equation}

We will now find the infimum in the above relation. Note that
\begin{multline} \label{eq:2p15}
G'(x) = 1-\left(1-\frac q p\right)x^{-\frac q p}\left(x-\frac{c}{p+1}\right)^{\frac q p} - x^{1-\frac q p} \left(\frac q p\right)\left(x-\frac{c}{p+1}\right)^{\frac q p -1} = \\
= 1 - \left(1-\frac q p\right)\left(1-\frac{c}{(p+1)\,x}\right)^{\frac q p} - \frac q p \left(1-\frac{c}{(p+1)\,x}\right)^{\frac q p -1}.
\end{multline}
We consider now the function
\[
H(t) = 1-\left(1-\frac q p\right)t^{\frac q p} - \frac q p\, t^{\frac q p -1},\quad t\in(0,1).
\]
Then $H'(t) = -t^{q/p-2}\left(1-\frac q p\right)\frac q p\, (t-1) > 0$, for every $t\in (0,1)$. Thus $H(t)$ is strictly increasing $\implies H(t) \leq H(1) = 0,\ \forall t\in(0,1)$. By setting now $t=1-\frac{c}{(p+1)x}$, we conclude that the expression in the right of \eqref{eq:2p15} is negative, that is $G'(x)\leq 0,\ \forall x>\frac{c}{p+1}\implies G$ is decreasing in $\left(\frac{c}{p+1},+\infty\right)$.
Thus $G(x) \geq \lim_{x\to+\infty}G(x) = \ell$.

Then $\displaystyle \ell = \lim_{x\to+\infty}\left[x-x^{1-\frac q p}\left(x-\frac{c}{p+1}\right)^{\frac q p}\right] = \lim_{x\to+\infty}\frac{1-\left(1-\frac{c}{(p+1)x}\right)^{\frac q p}}{\frac{1}{x}} = $\\ $= \lim_{y\to 0^+}\frac{1-\left(1-\frac{c}{p+1}y\right)^{\frac q p}}{y} = -\frac q p \left(-\frac{c}{p+1}\right) = \frac{q\,c}{p\,(p+1)}$, by applying the De L'Hospital rule. Thus we have by \eqref{eq:2p14} that $\displaystyle \left(\frac{p+1}{p}\right)^qz-x \geq p\frac{q\,c}{p\,(p+1)} = \frac{q\,c}{p+1}$, which gives inequality \eqref{eq:1p8}, by the definitions of $x$, $z$ and $c$.

The proof of Theorem \ref{thm:2} in the case $p>q$ is complete. The case $p=q$ is also true by continuity reasons, that is by letting $p\rightarrow q^+$
in \eqref{eq:1p8}.
\end{proof}

\begin{proof}[Proof of Theorem \ref{thm:1}] ~
We first prove the validity of \eqref{eq:1p7}. We simplify the proof by considering the case where $a=0$ and $b=1$.
We consider also the case where $f: [0,1]\to\mathbb R^+$ is continuous. The general case for Riemann integrable functions can be handled by using
approximation arguments which involve sequences of continuous functions. We suppose that $\int_0^1f=\ell$. We define $F: (0,1]\to\mathbb R^+$ by
$F(t)=\frac{1}{t}\int_0^tf(u)\,\mathrm du$. Then
$$\int_0^1\left(\frac{1}{t}\int_0^tf(u)\,\mathrm du\right)^{\mathrlap{-p}}\mathrm dt=\int_0^1(F(t))^{-p}\mathrm dt.$$

The integral above can be approximated by Riemann sums of the following type:
$$\sum_{n=1}^{2^k} \frac{1}{2^k}(F(\frac{n}{2^k}))^{-p}=\frac{1}{2^k}\sum_{n=1}^{2^k}(\frac{\sum_{i=1}^{n}a_i^{(k)}}{n})^{-p}.$$
where the quantities $a_i^{(k)}$ are defined as follows:

$$a_i^{(k)}=2^k\int_{\frac{i-1}{2^k}}^{\frac{i}{2^k}}f$$
for $i=1,..., 2^k$.
We use now inequality \eqref{eq:1p8}. Thus the sum that appears above is less or equal than

$$(\frac{p+1}{p})^q\frac{1}{2^k}\sum_{n=1}^{2^k}(\frac{\sum_{i=1}^{n}a_i^{(k)}}{n})^{-p+q}(a_n^{(k)})^{-q}
-\frac{q}{p+1}(\frac{\sum_{n=1}^{2^k}a_n^{(k)}}{2^k})^{-p}.$$

Now we obviously have that $\frac{\sum_{n=1}^{2^k}a_n^{(k)}}{2^k}=\ell$, while since $f$ is continuous, for every $n=1,..., 2^k$  there exists
$b_n^{(k)} \in [\frac{n-1}{2^k},\frac{n}{2^k}]$, such that $a_n^{(k)}=f(b_n^{(k)})$. Thus the quantity that appears above equals

$$(\frac{p+1}{p})^q\frac{1}{2^k}\sum_{n=1}^{2^k}(F(\frac{n}{2^k}))^{-p+q}(f(b_n^{(k)}))^{-q}-\frac{q}{p+1}\ell^{-p}.$$

Now, by continuity reasons, and by the choice of $b_n^{(k)}$, the quantity above approximates

$$(\frac{p+1}{p})^q\frac{1}{2^k}\sum_{n=1}^{2^k}(F(b_n^{(k)}))^{-p+q}(f(b_n^{(k)}))^{-q}-\frac{q}{p+1}\ell^{-p}$$
as $k\to\infty$. It is clear now that this last quantity approximates the right side of \eqref{eq:1p7}, as $k\to\infty$.

We now prove the sharpness of \eqref{eq:1p7}. Let $\ell>0$ be fixed and $p\geq q>0$. We consider for any $a\in\left(-\frac 1 p,0\right)$ the following function $g_a(t) = \ell\,(1-a)\,t^{-a}$, $t\in [0,1]$. It is easy to see that $\int_0^1g_a=\ell$, $\frac 1 t \int_0^tg_a=\frac{1}{1-a}g_a(t)$ for every $t\in (0,1]$ and that $\int_0^1g_a^{-p} = \frac{\ell^{-p}(1-a)^{-p}}{1+ap}$. We consider now the difference
\[
L_a = \int_0^1\left(\frac 1 t \int_0^tg_a\right)^{-\mathrlap{p}}\mathrm dt - \left(\frac{p+1}{p}\right)^q\int_0^1\left(\frac 1 t \int_0^tg_a\right)^{-p\mathrlap{+q}}g_a^{-q}(t)\,\mathrm dt.
\]
It equals to (because of the above properties that $g_a$ satisfy)
\[
L_a = \ell^{-p}\frac{\left[1-(1-a)^{-q}\left(\frac{p+1}{p}\right)^q\right]}{1+a\,p}.
\]

We let $a\to-\frac{1}{p}^+$ and we conclude that
\[
\lim_{a\to-\frac{1}{p}^+}L_a = \left.\ell^{-p}\,q\,(1-a)^{-q-1}\right]_{a=-\frac 1 p}(-1)\left(\frac{p+1}{p}\right)^q = -\frac{q}{p+1}\ell^{-p}.
\]

In this way we derived the sharpness or \eqref{eq:1p7}.

The proof of Theorem \ref{thm:1} is complete.
\end{proof}

\section{Proof of Theorem \ref{thm:3}} \label{sect:3}
Let $\varphi: [0,1)\to\mathbb R^+$ be non decreasing satisfying the inequality
\begin{equation} \label{eq:3p1}
\left(\frac 1 t\int_0^t \varphi\right)\left(\frac 1 t \int_0^t\varphi^{-1/(q-1)}\right)^{q-1} \leq M,
\end{equation}
for every $t\in(0,1]$, where $q$ is fixed such that $q>1$ and $M>0$. We assume also that there exists an $\varepsilon>0$ such that $\varphi(t)\geq \varepsilon>0$, $\forall t\in[0,1)$. The general case can be handled using this one, by adding a small constant $\varepsilon>0$ to $\varphi$.

We need the following from \cite{9}.
\begin{lemmaA} \label{lem:A}
Let $\psi:(0,1)\to[0,+\infty)$, such that $\lim_{t\to0}t\,[\psi(t)]^a=0$, where $a\in\mathbb R$, $a>1$ and $\psi(t)$ is continuous and monotone on $(0,1)$. Then the following is true for any $a\in(0,1)$.
\begin{equation} \label{eq:3p2}
a\int_0^u\psi^{a-1}(t)\,[t\,\psi(t)]'\mathrm dt = u\,\psi^a(u)+(a-1)\int_0^u\psi^a(t)\,\mathrm dt.
\end{equation}
\end{lemmaA}
We refer to \cite{9} for the proof.

We continue the proof of Theorem \ref{thm:3}. We set $h:[0,1) \to \mathbb R^+$ by $h(t) = \varphi^{-1/(q-1)}(t)$. Then obviously $h$ satisfies $h(t)\leq \varepsilon^{-1/(q-1)}$, $\forall t\in[0,1)$. Let also $p_0\in[1,q]$ be defined such that
\[
\frac{q-p_0}{q-1}\left(M\,p_0\right)^{1/(q-1)} = 1.
\]
Let also $p\in(p_0,q]$. Define $\psi$ by $\psi(t) = \frac 1 t \int_0^t\varphi^{-1/(q-1)}$. Then by Lemma \ref{lem:A}, we get for $a=\frac{q-1}{p-1}>1$, the following:
\begin{multline} \label{eq:3p3}
\frac{q-1}{p-1}\int_0^t\varphi^{-1/(q-1)}(s)\left(\frac 1 s \int_0^s \varphi^{-1/(q-1)}\right)^{\frac{q-p}{p-1}}\mathrm ds - \\
 -(\frac{q-p}{p-1})\int_0^t\left(\frac 1 s \int_0^s\varphi^{-1/(q-1)}\right)^{\frac{q-1}{p-1}}\mathrm ds = t\left(\frac 1 t \int_0^t\varphi^{-1/(q-1)}\right)^{\frac{q-1}{p-1}}.
\end{multline}

Define for every $y>0$ the following function of the variable of $x\in[y,+\infty)$
\begin{equation} \label{eq:3p4}
g_y(x) = \frac{q-1}{q-p} y\,x^{(q-p)/(p-1)}-x^{(q-1)/(p-1)}.
\end{equation}

Then $g'_y(x) = \frac{q-1}{p-1} x^{\left[(q-1)/(p-1)\right]-2}(y-x) \leq 0$, $\forall x\geq y$. Then $g_y$ is strictly decreasing on $[y,+\infty)$.

So if $y\leq x\leq w \implies g_y(x) \geq g_y(w)$. For every $s\in(0,t]$ set now
\[
x=\frac 1 s \int_0^s\varphi^{-1/(q-1)},\ y=\varphi^{-1/(q-1)}(s),\ c=M^{1/(q-1)},\ \text{and}\ z = \left(\frac 1 s \int_0^s \varphi\right)^{-\frac 1 {q-1}}.
\]
Note that by \eqref{eq:3p1} the following is true $y\leq x\leq c\,z =: w$. Thus
\begin{multline} \label{eq:3p5}
g_y(x) \geq g_y(w) \implies \\
\frac{q-1}{q-p} \varphi^{-1/(q-1)}(s) \left(\frac 1 s \int_0^s\varphi^{-1/(q-1)}\right)^{\frac{(q-p)}{(p-1)}} - \left(\frac 1 s \int_0^s\varphi^{-1/(q-1)}\right)^{\frac{(q-1)}{(p-1)}} \geq \\
\geq \frac{q-1}{q-p} \varphi^{-1/(q-1)}(s)\left(\frac 1 s \int_0^s \varphi\right)^{\frac{1}{q-1}-\frac{1}{p-1}}c^{\frac{q-p}{p-1}} - c^{\frac{q-1}{p-1}}\left(\frac 1 s \int_0^s\varphi\right)^{-\frac{1}{(p-1)}}
\end{multline}

Integrating \eqref{eq:3p5} on $s\in(0,t]$ we get
\begin{multline} \label{eq:3p6}
\frac{q-1}{q-p} \int_0^t\varphi^{-1/(q-1)}(s) \left(\frac 1 s \int_0^s \varphi\right)^{-\frac{1}{p-1}+\frac{1}{q-1}}\mathrm ds\cdot c^\frac{q-p}{p-1}\leq \\
\leq \frac{q-1}{q-p}\int_0^t\varphi^{-1/(q-1)}(s)\left(\frac 1 s \int_0^s\varphi^{-1/(q-1)}\right)^{\frac{q-p}{p-1}}\mathrm ds - \\
- \int_0^t\left(\frac 1 s \int_0^s \varphi^{-1/(q-1)}\right)^{\frac{q-1}{p-1}}\mathrm ds + c^{\frac{q-1}{p-1}}\int_0^t \left(\frac 1 s \int_0^s \varphi\right)^{-1/(p-1)}\mathrm ds
\end{multline}

Now because of \eqref{eq:3p3} we get
\begin{multline} \label{eq:3p7}
\frac{q-1}{q-p} \int_0^t\varphi^{-1/(q-1)}(s) \left(\frac 1 s \int_0^s \varphi^{-1/(q-1)}\right)^{\frac{q-p}{p-1}}\mathrm ds - \int_0^t\left(\frac 1 s \int_0^s \varphi^{-1/(q-1)}\right)^\frac{q-1}{p-1}\mathrm ds \\
= \frac{p-1}{q-p}\, \frac{1}{t^{(q-p)/(p-1)}} \left(\int_0^t \varphi^{-1/(q-1)}\right)^{\frac{q-1}{p-1}}
\end{multline}

Thus \eqref{eq:3p6} gives
\begin{multline} \label{eq:3p8}
c^\frac{q-p}{p-1} \frac{q-1}{q-p} \int_0^t \varphi^{-1/(q-1)}(s) \left(\frac 1 s \int_0^s \varphi\right)^{-\frac{1}{p-1}+\frac{1}{q-1}}\mathrm ds \leq \\
\leq c^\frac{q-1}{p-1} \int_0^t \left(\frac 1 s \int_0^s \varphi\right)^{-1/(p-1)}\mathrm ds + \frac{p-1}{q-p}\,t\left(\frac 1 t \int_0^t\varphi^{-1/(q-1)}\right)^{(q-1)/(p-1)}.
\end{multline}

But
\begin{multline} \label{eq:3p9}
\left[\frac 1 t \int_0^t \varphi^{-1/(q-1)}\right]^{(q-1)/(p-1)} \leq M^{1/(p-1)} \left(\frac 1 t \int_0^t \varphi\right)^{-1/(p-1)} \overset{\eqref{eq:3p8}}{\implies} \\
\begin{multlined}[\textwidth]
c^{\frac{q-p}{p-1}}\,\frac{q-1}{q-p} \int_0^t\varphi^{-1/(q-1)} (s) \left(\frac 1 s \int_0^s \varphi\right)^{-\frac{1}{p-1}+\frac{1}{q-1}}\mathrm ds \leq \\
\leq c^{\frac{q-1}{p-1}} \int_0^t \left(\frac 1 s \int_0^s\varphi\right)^{-1/(p\mathrlap{-1)}}\mathrm ds + \frac{p-1}{q-p}\,t\,M^{1/(p-1)}\left(\frac 1 t\int_0^t\varphi\right)^{-1/\mathrlap{(p-1)}} \implies
\end{multlined} \\
\begin{multlined}[\textwidth]
A_1:=\frac{q-1}{q-p} \int_0^t\varphi^{-1/(q-1)}(s)\left(\frac 1 s \int_0^s\varphi\right)^{-\frac{1}{p-1}+\frac{1}{q-1}}\mathrm ds \leq \\
\leq c \int_0^t \left(\frac 1 s \int_0^s\varphi\right)^{-1/(p-1)}\mathrm ds + \frac{p-1}{q-p} \frac{M^{1/(p-1)}}{c^{(q-p)/(p-1)}}\,t\left(\frac 1 t \int_0^t\varphi\right)^{-1/(p-1)}.
\end{multlined}
\end{multline}

Now by using Theorem \ref{thm:1} we get
\begin{multline} \label{eq:3p10}
\int_0^t\left(\frac1 s \int_0^s\varphi\right)^{-\frac{1}{p-1}}\mathrm ds \leq \\
 \left(\frac{1+\frac{1}{p-1}}{\frac{1}{p-1}}\right)^{\frac{1}{q-1}}\hspace{-.5em}\int_0^t\left(\frac 1 s \int_0^s \varphi\right)^{-\frac{1}{p-1}+\frac{1}{q-1}}\hspace{-.5em}\varphi^{-\frac{1}{q-1}}(s)\,\mathrm ds -
\frac{\frac{1}{q-1}}{1+\frac{1}{p-1}}\,t\left(\frac 1 t \int_0^t\varphi\right)^{-\mathrlap{\frac{1}{p-1}}} = \\
= p^{\frac{1}{q-1}} A_1\,\frac{q-p}{q-1} - \frac{p-1}{(q-1)\,p}\,t\left(\frac 1 t \int_0^t\varphi\right)^{-\frac{1}{p-1}}.
\end{multline}

Thus in view of \eqref{eq:3p10}, \eqref{eq:3p9} becomes
\begin{multline} \label{eq:3p11}
\begin{multlined}[.95\textwidth]
A_1 \leq c\,p^{1/(q-1)}A_1\,\frac{q-p}{q-1} - c\,\frac{p-1}{(q-1)\,p}\,t\left(\frac 1 t \int_0^t\varphi\right)^{-1/(p-1)} + \\
+ \frac{p-1}{q-p}\, \frac{M^{1/(p-1)}}{c^{(q-p)/(p-1)}}\, t \left(\frac 1 t \int_0^1\varphi\right)^{-1/\mathrlap{(p-1)}} \implies
\end{multlined} \\
\begin{multlined}[.95\textwidth]
\left[1-c\,p^{1/(q-1)}\frac{q-p}{q-1}\right]A_1 \leq \left[\frac{M^{1/(p-1)}}{c^{(q-p)/(p-1)}}\,\frac{p-1}{q-p} - c\,\frac{p-1}{(q-1)\,p}\right]\cdot \\
\cdot t \left(\frac 1 t \int_0^t\varphi\right)^{-1/\mathrlap{(p-1)}} \implies
\end{multlined} \\
\begin{multlined}[.95\textwidth]
\implies K(p,q,c) \left[\frac 1 t \int_0^t\varphi^{-1/(q-1)}(s) \left(\frac 1 s \int_0^s\varphi\right)^{-1/(p-1)+1/(q-\mathrlap{1)}}\mathrm ds \right] \leq \\
\leq \left[\frac{p-1}{q-1}\,\frac{M^{1/(p-1)}}{c^{(q-p)/(p-1)}} - c\,\frac{(p-1)(q-p)}{p\,(q-1)^2}\right] \left(\frac 1 t \int_0^t\varphi\right)^{-1/(p-1)}
\end{multlined}
\end{multline}
where $K=K(p,q,c) = 1-c\,p^{1/(q-1)}\frac{q-p}{q-1}>0$, $\forall p\in(p_0,\,q]$.

As a consequence \eqref{eq:3p11} gives
\begin{multline} \label{eq:3p12}
K\left[\frac 1 t \int_0^t\varphi^{-1/(p-1)}(s)\left(\frac 1 s \int_0^s\varphi\right)^{-1/(p-1)+1/(q-1)}\mathrm ds\right] \leq \\
\leq \left(\frac 1 t \int_0^t\varphi\right)^{-1/(p-1)}\left(\frac{p-1}{q-1}\right)^2 c\,\frac q p.
\end{multline}

Now we use the inequality
\begin{multline*}
\frac 1 t \int_0^t\varphi^{-1/(q-1)}(s)\left(\frac 1 s \int_0^s\varphi\right)^{-1/(p-1)+1/(q-1)}\mathrm ds  \geq \\
\geq \left[\frac{1/(p-1)}{1+(1/(p-1))}\right]^{1/(q-1)} \cdot \frac 1 t \int_0^t\left(\frac 1 s \int_0^s\varphi \right)^{-1/(p-1)}\mathrm ds
\end{multline*}
which is true because of Theorem \ref{thm:E}. Thus \eqref{eq:3p12} gives
\begin{equation} \label{eq:3p13}
\frac{K'}{t}\int_0^t\left(\frac 1 s \int_0^s\varphi\right)^{-1/(p-1)}\mathrm ds \leq \left(\frac 1 t \int_0^t\varphi\right)^{-1/(p-1)}\left(\frac{p-1}{q-1}\right)^2c\,\frac q p
\end{equation}
where $K'=\frac{K}{p^{1/(q-1)}}$, $K = 1-c\,p^{1/(q-1)}\frac{q-p}{q-1}$.

Thus the inequality stated in Theorem \ref{thm:3} is proved.

We need to prove the sharpness of \eqref{eq:3p13}. We consider $a$ such that $0<a<q-1$ and the function $\varphi_a: (0,1]\to\mathbb R^+$ defined by $\varphi_a(t) = t^a$, $t\in(0,1]$. The function $\varphi_a$ is strictly increasing and $\frac 1 t \int_0^t\varphi_a = \frac 1 t \frac{t^{a+1}}{a+1} = \frac{1}{a+1}\varphi_a(t)$, $\forall t\in(0,1]$, while $\int_0^t\varphi_a^{-1/(q-1)} = \frac{1}{1-a/(q-1)} t^{1-a/(q-1)}$. Thus
\begin{multline*}
\left(\frac 1 t \int_0^t\varphi_a\right)\left[\frac 1 t \int_0^t \varphi_a^{-1/(q-1)}\right]^{q-1} = \left[\frac{q-1}{q-1-a}\right]^{q-1}\left[t^{-a/(q-1)}\right]^{q-1}\cdot \\
\cdot\left(\frac 1 t \int_0^t\varphi_a\right) = \frac{1}{a+1}\left(\frac{q-1}{q-1-a}\right)^{q-1} =: M(q,a)
\end{multline*}
and
\[
c_a = c(q,a) = [M(q,a)]^{1/(q-1)} = \left[\frac{q-1}{(q-1)-a}\right]\frac{1}{(1+a)^{1/(q-1)}}.
\]

Let now $p\in (p_0,q]$ and suppose additionally that $a<p-1$ so that $\int_0^1\varphi_a^{-1/(p-1)} = (p-1)/(p-1-a)$. We prove the sharpness of \eqref{eq:1p11} for $t=1$. That is we prove that the inequality
\[
\frac{K'}{t}\int_0^t\left(\frac 1 s \int_0^s\varphi\right)^{-1/(p-1)}\mathrm ds \leq \left(\frac 1 t \int_0^t\varphi\right)^{-1/(p-1)}c\,\frac q p \left(\frac{p-1}{q-1}\right)^2
\]
becomes sharp for $t=1$.
Obviously if $\displaystyle I_a=\int_0^1\left(\frac 1 s \int_0^s\varphi_a\right)^{-1/(p-1)}\mathrm ds$, then $\displaystyle I_a = \frac{1}{(1+a)^{-1/(p-1)}}\cdot$ $\cdot\int_0^1\varphi_a^{-1/(p-1)} = (1+a)^{1/(p-1)}\frac{1}{1-a/(p-1)}$ while $\displaystyle \left(\int_0^1\varphi_a\right)^{-1/(p\mathrlap{-1)}} = \left(\frac{1}{a+1}\right)^{-1/(p-1)}$. Thus in order to prove the sharpness of the above inequality we just need to prove that the following is true
\begin{multline} \label{eq:3p14}
\left[\frac{1}{p^{1/(q-1)}}-\frac{q-p}{q-1}\,c_a\right]\left(\frac{p-1}{(p-1)-a}\right) \cong  c_a\,\frac q p \left[\frac{(p-1)}{(q-1)}\right]^2\ \text{as}\ a\to(p-1)^- \iff \\
\begin{multlined}[.95\textwidth]
\left[\frac{1}{p^{1/(q-1)}} - \frac{q-p}{q-1}\frac{1}{(1+a)^{1/(q-1)}}\left(\frac{q-1}{(q-1)-a}\right)\right]\frac{1}{(p-1)-a} \cong \\
\cong \frac q p \,\frac{p-1}{(q-1)^2}\,\frac{1}{(1+a)^{1/(p-1)}}\,\frac{q-1}{(q-1)-a}, \text{as}\ a\to(p-1)^-.
\end{multlined}
\end{multline}
Let then $a\to(p-1)^-$ or equivalently $x:=(a+1)\to p^-$. Then for the proof of \eqref{eq:3p14} we just need to note that
\[
\frac{\left[p^{-\frac{1}{1/(q-1)}} - \frac{q-p}{q-x}\,\frac{1}{x^{1/(q-1)}}\right]}{p-x} \cong \frac q p\, \frac{p-1}{q-1}\,\frac{1}{p^{1/(q-1)}}\,\frac{1}{q-p},\ \ \text{as}\ \ x\to p^-,
\]
which is a simple application of De L'Hospitals rule.

The proof of Theorem \ref{thm:3} is now complete.

Nikolidakis Eleftherios, University of Ioannina, Department of Mathematics, E-mail address: enikolid@cc.uoi.gr

Stavropoulos Theodoros, National and Kapodistrian University of Athens, Department of Mathematics, E-mail address: tstavrop@math.uoa.gr


\end{document}